\documentclass[11pt]{amsart}

\usepackage{fullpage}
\usepackage{hyperref}
\usepackage{amsthm,amsfonts,amssymb,amsmath,oldgerm}
\numberwithin{equation}{section}
\usepackage{caption,graphicx}
\usepackage{paralist}

\usepackage{hyperref}

\newtheorem{theorem}{Theorem}[section]
\newtheorem{theo}{Theorem}[section]

\newtheorem{lem}[theorem]{Lemma}

\newtheorem{prop}[theorem]{Proposition}

\newtheorem{cor}[theorem]{Corollary}

\def\eps{\varepsilon }

\newcommand{\RR}{\mathbb{R}}

\newcommand{\NN}{{\mathbb N}}
\newcommand{\ZZ}{{\mathbb Z}}
\newcommand{\TT}{{\mathbb T}}

\def\beq{\begin{equation}}
\def\eeq{\end{equation}}
\def\bb1{{1\!\!1}}
\def\eps{\varepsilon}

\begin{document}

\title{Linear inviscid damping and enhanced viscous dissipation of  shear flows by using the conjugate operator method}

\author{Emmanuel Grenier \and Toan T. Nguyen \and Fr\'ed\'eŽric Rousset \and Avy Soffer}

\address[E. Grenier]{Equipe Projet Inria NUMED,
 INRIA Rh\^one Alpes, Unit\'e de Math\'ematiques Pures et Appliqu\'ees., 
 UMR 5669, CNRS et \'Ecole Normale Sup\'erieure de Lyon,
               46, all\'ee d'Italie, 69364 Lyon Cedex 07, France.}
               \email{Emmanuel.Grenier@ens-lyon.fr}

\address[T. Nguyen]{Department of Mathematics, Penn State University, State College, PA 16803.}
\email{nguyen@math.psu.edu}

\address[F. Rousset]{Laboratoire de Math\'ematiques d'Orsay (UMR 8628), Universit\'e Paris-Sud et Institut Universitaire de France, 91405 Orsay Cedex, France}
\email{frederic.rousset@math.u-psud.fr}

\address[A. Soffer]{Department of Mathematics, Rutgers University, New Brunswick, NJ 08903} \email{soffer@math.rutgers.edu}

\maketitle


\begin{abstract}


We study the large time behavior of solutions to two-dimensional Euler and Navier-Stokes equations linearized about shear flows of the mixing layer type in the unbounded channel $\TT \times \RR$.
 Under a simple spectral stability assumption on a self-adjoint operator, we prove 
  a local form of the  linear inviscid damping that is uniform with respect to small viscosity.
   We also prove a local form of the enhanced viscous dissipation that takes place 
  at times of  order $\nu^{-1/3},$  $\nu$ being the small viscosity. 
  To prove these results, we use a Hamiltonian approach, following the conjugate operator method developed in the study of Schr\"odinger operators,  combined with a hypocoercivity argument to handle 
   the viscous case. 

\end{abstract}




\section{Introduction}


In this paper, we are interested in the long time behavior of solutions to the two-dimensional incompressible
 Euler and  Navier Stokes equations in the vanishing viscosity limit linearized about 
near a stationary shear flow. 
More precisely, we shall study the following linearized incompressible Euler and Navier-Stokes systems respectively, 
\beq
\label{Euler1}
\partial_t v + (U_s \cdot \nabla) v + (v \cdot \nabla) U_s +\nabla p = 0, \quad \nabla \cdot v=0,
\eeq
and 
\beq \label{NS1}
\partial_t v + (U_s \cdot \nabla) v + (v \cdot \nabla) U_s +   \nabla p  - \nu \Delta v = 0, \quad \nabla \cdot v= 0, 
\eeq
for  $(x,y) \in \TT \times \RR$   and  $\nu>0$ small, 
 where $U_s$ is a smooth stationary shear flow under  the form 
\begin{equation}\label{shear}
U_s = \left( \begin{array}{c} U(y) \cr
0 \cr \end{array} \right).
\end{equation}

\bigskip

We shall study  at this linearized  level the  inviscid damping and  the  enhanced viscous dissipation. 
 In the particular case of the Couette flow, $U(y)= y$, these are well known phenomena that go back to 
observations by  Kelvin and Orr in fluid mechanics and can be justified from an explicit computation in the Fourier space.
  By denoting by $(v^1_{\alpha}, v^2_{\alpha})_{\alpha \in \mathbb{Z}}$ the Fourier coefficients
  (taking  Fourier series  in the $x$ variable) of the velocity $v=(v^1,v^2)$,  the inviscid damping  is
    the property that for $\alpha \neq 0$, and for smooth enough initial data,  we have in the large time for the solution of \eqref{Euler1}
    \beq
    \label{decayintro1} \|v^1_{\alpha} (t)\| \lesssim {1 \over \alpha t}, \qquad   \|v^2_{\alpha} (t)\| \lesssim {|\alpha | \over (\alpha t)^2}
    \eeq
    where throughout the paper $\| \cdot \|$ will stand for the $L^2$ norm in the $y\in \mathbb{R}$ variable.
     The main reason for this decay is the mixing phenomenon produced by the free transport operator
      $y \partial_{x}$.
 This property  is also true for the solution of  \eqref{NS1} uniformly with respect to $\nu$.
  The enhanced dissipation is the property that for the solution of \eqref{NS1}, we have
  \beq
  \label{decayintro2} \| \omega_{\alpha} (t)\| \lesssim e^{- \nu  t^3}
  \eeq
for $\omega_{\alpha}$ denoting the Fourier coefficients
   of the vorticity $\omega= \partial_{x}v^2 - \partial_{y} v^1$. This shows that the solution of \eqref{NS1} is damped  by  the combination of mixing and viscosity at the time scale
   $\nu^{- {1 \over 3 }}$ which is much shorter than  the viscous time scale
   $\nu^{-1}$.
    For more details, we refer to the introduction of \cite{BedMVV}. Note that   outstanding results 
    that prove that these properties are still true for solutions of  the nonlinear equations 
     close to  the Couette flow in strong enough norms have been obtained recently \cite{BedM15, BedMVV, BedVVWang}. 
The enhanced viscous dissipation makes use of the hypocoercivity of the transport-diffusion equation; see, for instance, \cite{Beck,BedCZ,Villani}. 

\bigskip
    
     The generalization of these properties to nontrivial shear flows has also received a lot of attention.
     Small in some sense perturbations of the  Couette flow  have been studied in \cite{Z16},  
      the case of  (possibly degenerate) monotonic shear flows in bounded channels
       (that is to say in $\mathbb{T} \times [0, 1]$) has been studied
        in   \cite{Zhang17a,Zhang17b} and the Kolmogorov flow that is to say the shear flow
         $U(y)= \sin y$ in a doubly periodic channel has been studied in \cite{LinXu, Zhang17c, MaeMas}.
          The case of  radial vortices has been also much studied recently  \cite{BedCZVV, Gallay, Li-Wei-Zhang}. 

\bigskip

In this paper, we shall focus on mixing layers type shear flows  in $\mathbb{T} \times \mathbb{R}$. Precisely, 
 we assume that $U$ is smooth and satisfies 
 \begin{equation}\tag{\bf H1}
  \forall ~y\in \RR, \, U'(y)>0, \quad \lim_{y\to \pm \infty} U(y) = U_\pm, \quad 
 {U''\over U} \in L^\infty,  \quad \forall ~y\in \RR, \,  {U'' \over U} (y)<0
  \end{equation}
for some constants $U_\pm$.
 Let us comment on the above  assumptions on $U''/U$. A smooth shear flow that satisfies the two first properties
 necessarily has an inflexion point.  In view of Rayleigh's inflexion point theorem, we therefore have to be careful in order to ensure its linear stability. The classical shear flows for which this can be ensured are the shear flows of the so-called
  $\mathcal{K}_{+}$ family for which we assume that there exists a unique inflexion point $y_{s}$
   and that $- U''(y)/ (U(y) -U(y_{s}))$ is bounded and positive. By changing $x$ into $x-ct$, with $c= U(y_{s})$, 
    we can always change  $U$ into $U-U(y_{s})$
     in \eqref{Euler1}  and \eqref{NS1} so that the two  last assumptions in {\bf (H1)} are verified.
   We will also make the following mild assumption. Let us set $m= (- U''/U)^{1\over 2}$  which  is well defined
    (as a real positive function)
    and smooth thanks to {\bf(H1)}. Assume that
  \begin{equation}\tag{\bf H2}  \forall k, \exists C_{k}>0, \quad  |m^{(k)}| \leq C_{k} m, \quad \lim_{|y| \rightarrow +\infty} m= 0, \quad
  m \in L^2, \quad {U'' \over U-U_{\pm}} \in L^\infty\cap L^2,
  \end{equation} 
with $m^{(k)}$ being the $k^{th}$-order derivatives of $m$. 

\bigskip

Finally, we will make an assumption that ensures the spectral stability of $U$ for \eqref{Euler1}. This means
   that it excludes the existence of nontrivial solutions of \eqref{Euler1} such that
   \beq
   \label{growingmode} \omega (t, x, y)= e^{\lambda t} e^{i \alpha x} \Omega(y), \quad \alpha \in \mathbb{Z}^*, \quad {\mbox {Re }}\lambda >0, \quad \Omega \in L^2(\mathbb{R}).
   \eeq
   Note that of course in the presence of such instabilities estimates like \eqref{decayintro1} cannot be true.
    Let us consider the Schr\"odinger operator
    \beq
    \label{Lcurldef} \mathcal{L}= - \partial_{y}^2  -  m^2, \qquad  D(\mathcal{L})= H^2(\mathbb{R})
    \eeq
    and define $\lambda_{0}$ as the infimum of the spectrum of this self-adjoint operator
     (note that because of {\bf (H2)}, its essential spectrum is $[0, + \infty[$). We assume
   \begin{equation}\tag{\bf H3}  \lambda_{0}>-1.
   \end{equation} 
  Note that this assumption is almost sharp. Indeed, if $\lambda_{0}<-1$, we get from Theorem 1.5 of
   \cite{LinSIAM} that there exist growing modes of the form \eqref{growingmode}
    for every $\alpha \in (0, \sqrt{-\lambda_{0}})$ and in particular for $\alpha = 1$ so that
    $U$ is unstable on $\mathbb{T}\times \mathbb{R}.$
    
    \bigskip
    
    The main examples of shear profiles $U(y)$ for which assumptions {\bf(H1)-(H3)} are verified are shear flows under the form
     $$U(y)= V\left({y\over L}\right)$$ where we can take $V$ under the form
     $$ V(z)=  \tanh z, \quad \mbox{or} \quad V(z)=  \int_{0}^z { 1 \over (1 + s^2)^k}\, ds$$
     for $k$ sufficiently large. Assumptions {\bf(H1)} and {\bf (H2)} are easily verified, while Assumption {\bf (H3)} is
       verified if $L$ is sufficiently large. In the case of the hyperbolic tangent the lowest eigenvalue
        of $\mathcal{L}$ is explicitly known; precisely, we have $\lambda_{0}(L)=-  {1 \over L^2}$ (the associated eigenfunction
         being $1/\cosh (y/L)$). Hence, {\bf (H3)} is verified as soon as $L>1$. Again,  this is sharp, since if $L<1$, we get from \cite{LinSIAM}
          that the mixing layer is unstable.

\bigskip      
  
  The aim of this paper is to show that for shear flows satisfying {\bf (H1)-(H3)} appropriate local versions
   of \eqref{decayintro1} and \eqref{decayintro2} hold. One of the main purposes of this paper is also  to introduce an Hamiltonian approach to prove the inviscid damping, with sharp decay in time  following the conjugate operator method, which has been well developed in the study of Hamiltonian operators; for instance, see \cite{ABG,BGS,GLS,HSS}.
    This approach  is  different from the one in \cite{Zhang17a,Zhang17b, MaeMas}  where shear flows
    in bounded channels are considered. The approach was  
     based on a direct proof of the limiting absorption principle from resolvent constructions.
      This is also different from the approach of \cite{LinXu} that  relies more on an abstract argument
       like the RAGE theorem and gives qualitative results. Here,  
       we are able to get sharp quantitative estimates. 
        Our approach will rely on a suitable symmetrized version of the linearized Euler equation
        in vorticity form that we introduce in the next section.

\bigskip
         
    The paper is organized as follows. In the two next Sections, we describe our mains results.
    Sections \ref{proofEuler} and \ref{proofNS} are devoted to the proof of the main results.
    Finally, Section \ref{sectiontech} is devoted to the proof of some technical lemmas. 

\bigskip
Throughout the paper we use the notation $\|\cdot \|$ for the $L^2(\mathbb{R})$ norm and
$\langle  \cdot, \cdot \rangle$ for the real $L^2$ scalar product:
$$  \langle f, g \rangle= \mbox{Re } \int_{\mathbb{R}} f(y) \overline{g}(y) \, dy.$$
We also use the notation $\langle A \rangle= ( 1 + A^2)^{1 \over 2}$ for symmetric operator $A=i\partial_y$ on $L^2$. In addition, for $\alpha \in \ZZ$, we write $\nabla_\alpha = (\partial_y, i\alpha)^T$ and $\Delta_\alpha = \partial_y^2 - \alpha^2$. 

\bigskip

\section{Inviscid damping}
\label{introEuler}

\subsection{Symmetric form of the equation}

 We shall work with the vorticity form of the equation \eqref{Euler1}.
 Set $\omega= \partial_{x}v^2 - \partial_{y} v^1$, then $\omega$ solves
 $$ \partial_{t} \omega + U(y)\partial_{x} \omega -  v^2 U''(y)= 0$$
  and $v^2$ can be recovered from $\omega$ by $ \Delta v^2= \partial_{x} \omega$.
  
   Let $\alpha\in \ZZ$ be the corresponding Fourier variable of $x$. Taking the Fourier transform in $x$, we rewrite the above  equation   in the Fourier space as:  
\begin{equation}\label{Euler-a} i\partial_t \omega_\alpha = \alpha L_0(y) \omega_\alpha, \qquad L_0 := U(y)  -  U''(y) \Delta_\alpha^{-1} , \end{equation}
in which $\Delta_\alpha = \partial_y^2 - \alpha^2$. When $\alpha =0$, the problem is reduced to $\partial_t \omega_\alpha =0$ and therefore no mixing occurs.  We shall thus  only consider the case when $\alpha \not =0$. We shall moreover focus on the case $\alpha >0$. 
 The case $\alpha <0$ can be handled from the same arguments as below by reversing the direction of propagation.
  It is then convenient to use a change of time scale in \eqref{Euler-a}, we set 
 \beq
 \label{change time} \omega_{\alpha} (t,y) = \tilde{\omega}_{\alpha}( \alpha t, y)
 \eeq so that dropping the tilde and the subscript $\alpha$, we obtain
   \begin{equation}\label{Euler-a'} i\partial_t \omega= L_0(y) \omega, \qquad L_0 = U(y)  -  U''(y) \Delta_\alpha^{-1}. \end{equation}
For convenience, we write $$L_0 =  U(y) \Big ( 1 + m(y)^2  \Delta_{\alpha}^{-1} \Big) , \qquad m(y):= (- U''(y)/U(y))^{1\over 2}.$$

      Let us introduce the operator 
      $$  \Sigma =  1 + m \Delta_{\alpha}^{-1} m.$$
      Note that $\Sigma$ depends on $\alpha$, through $\Delta_\alpha^{-1}$, but we omit to write this dependence explicitly.
       We observe that $\Sigma$ is  a bounded symmetric operator on $L^2$ and that $m \Delta_{\alpha}^{-1} m$ is a compact operator, upon noting that
        $m$ tends to zero at infinity thanks to {\bf(H2)}.  Moreover,  mainly thanks to {\bf(H3)}, we also have the following lemma whose proof is given in Section \ref{proofEuler}. 
        
        \begin{lem}
        \label{lemSigma}
        Assuming  {\bf(H1)-(H3)}, there exists a constant $c_{0}>0$ such that for every $\alpha \in \mathbb{Z}^*$, in the sense of symmetric operators, 
         we have
        $$ \Sigma \geq c_{0}>0.$$ 
        \end{lem}
        
   We can thus write $\Sigma = S^2$ for some bounded symmetric coercive operator $S$ on $L^2$. Moreover, we also have that $S- 1$ is compact since
   $$ S-1= m \Delta_{\alpha}^{-1} m  ( 1 +  S)^{-1}.$$
    By setting 
    $ \omega= m S^{-1}  \psi$, we finally  find  
  \begin{equation}\label{symm-vort} 
  i \partial_{t} \psi =  H \psi, \quad  H=  S U(y) S, \quad S= \left (1 + m \Delta_{\alpha}^{-1} m \right)^{1 \over 2}
  \end{equation} 
      with the initial condition $\psi_{/t=0}= \psi_{0}= Sm^{-1} \omega_{0}$. Note that 
      we will always assume that  $\psi_{0} \in L^2$, which in terms of $\omega_{0}$ means that 
      $\omega_{0}$ is decaying sufficiently fast so that
        $ {1 \over m }\omega_{0} \in L^2$ 
       
     We also point out that $H$ is a bounded symmetric operator on $L^2$ and that $H$ actually depends 
      on $\alpha$ in a smooth way (since we focus on $|\alpha | \geq 1$). We omit this dependence for notational convenience. All the estimates that we shall give in the following
      are uniform with respect to $\alpha$.

\bigskip

\subsection{Spectral properties of H and conjugate operator} Let $ \sigma(H)$ be the spectrum of $H$ on $L^2$. 
The first useful property is that: 
\begin{lem}
\label{lemspectrum}
Assuming {\bf(H1)-(H3)}, we have $ \sigma(H)= [U_{-}, U_{+}]$ and there is no embedded eigenvalue in $[U_{-}, U_{+}]$.
\end{lem}
Again the proof of Lemma \ref{lemspectrum} will be given in Section \ref{proofEuler}.
To exclude eigenvalues and embedded eigenvalues we will  adapt   the arguments of \cite{LinSIAM, LinXu, Sattinger} for  the Rayleigh equation in bounded domains.

As an immediate Corollary,  we get from  the abstract RAGE Theorem that
\begin{cor}
 For any compact operator $C$ on $L^2(\mathbb{R})$, there holds
 $$  \lim_{T\rightarrow + \infty}{1 \over T} \int_{0}^T \|C e^{-i t H} \psi_{0} \|^2 \, dt = 0$$
 for any $\psi_{0} \in L^2$.
 
\end{cor}
In terms of the original vorticity function $\omega$, since $\omega = m S^{-1} \psi(t)$ with $S$ being a bounded operator, we observe that for every $\eps>0$, 
 the operator $C = \langle i\partial_y\rangle^{- \eps} m S^{-1}$ is compact on $L^2$, and hence the above result gives
$$   \lim_{T\rightarrow + \infty}{1 \over T} \int_{0}^T \|  \langle i\partial_y\rangle^{- \eps} \omega(t) \|^2 \, dt = 0$$
 for every $\omega_{0}$ such that ${1 \over m}\omega_{0} \in L^2$,  where $\omega$ solves  \eqref{Euler-a'}. In particular, this yields  some
 sort of time decay for the velocity.

We shall now  use the conjugate operator method to get  quantitative and more precise versions of this result.
We will use $ A =  i  \partial_{y}$ as a conjugate operator in order to exploit that $U'>0$.
 Note that $A$ is a symmetric operator on $L^2$.
   
   Observe that $ U: \mathbb{R} \rightarrow ]U_{-}, U_{+}[$ is a diffeomorphism. We can thus define a smooth function  $F$ on $]U_{-}, U_{+}[$  by
     $F(u)= U'( U^{-1}(u))$. This yields $F(U(y))= U'(y),$ for all $y \in \mathbb{R}$.   Note that for every compact interval $I \subset ]U_{-}, U_{+}[ $, there exists $\theta_{I}>0$
      such that $F \geq \theta_{I} $ on $I$.
    
      The crucial property that we will prove in Section \ref{proofEuler} is the following: 
      \begin{lem}
      \label{lemmourre} Assume {\bf (H1)-(H3)}.
       For every compact interval  $I \subset ]U_{-}, U_{+}[$,  there exists 
        a compact operator $K$ such that  for every  $g \in \mathcal{C}^\infty_{c} (]U_{-}, U_{+}[, \mathbb{R}_{+})$ with   the support contained in $I$, there holds
       \beq
       \label{estmourre}  g(H) i [H, A] g(H) \geq \theta_{I} g(H)^2 + g(H) K g(H)\eeq
       where $A = i\partial_y$, $[H,A] = HA - AH$, $\theta_{I}= \min_{I} F(u)>0$, and 
       $g(H)$ is defined through the usual functional calculus. 
      \end{lem}
       The above lemma is also true with  $g(H) = \mathrm{1}_{I}(H)$ the spectral projection onto $I$. We have stated the estimate in this way since it will be the one that
is  the most useful for us. 
        Note that in our simple setting, the commutator $ [H, A] $  and the higher iterates are bounded operators. 
       
       This localized commutator estimate was introduced in \cite{Mourre} and is well known to have many interesting consequences on the structure of the spectrum of $H$. Since we know from Lemma \ref{lemspectrum}
       that there are no  eigenvalues, we can get for example from \cite{Mourre, GerardC, ABG} that the limiting absorption principle holds
        for every interval $I \subset ]U_{-}, U_{+}[$ and that there is no singular continuous spectrum.
         Note that when $m \in L^2$, the operator $m \Delta_{\alpha}^{-1}m$ is in the trace class. This follows directly from the expression of the kernel which is given by 
         $$\mathcal{K}(y_{1}, y_{2})=  {1\over |\alpha|} e^{ - |\alpha| \,| y_{1} - y_{2}|  } m(y_{1}) m(y_{2}).$$
Thus, $H$ is a trace class perturbation of the multiplication operator by $U$. In addition, it follows from Kato's Theorem \cite{Kato}  that the continuous spectrum of $H$ is $\sigma_{ac}(H)= [U_{-}, U_{+}]$. By combining these facts, we get, again 
           from the Kato's theorem,  that
           the wave operators exist and are complete, which in particular implies the following 
           scattering result:
       \begin{cor} Assuming  {\bf (H1)-(H3)}, 
        for every $\psi_{0}\in L^2$, there exists $\psi_{+} \in L^2$ such that
           $$ \lim_{t \rightarrow + \infty} \| e^{-it H} \psi_{0} - e^{-it U(y)} \psi_{+}\| = 0.$$ 
           \end{cor}     
       
       Again, this can be translated into a scattering result in the original unknowns in a  weighted $L^2$ space.
              In the following we shall focus on the consequences of Lemma \ref{lemmourre} on time dependent 
       quantitative propagation estimates that are more flexible and in particular that can be 
        also performed for \eqref{NS1} for small positive $\nu$.
\bigskip

     \subsection{Main inviscid result}
     Our main result for \eqref{symm-vort} is the following:
     \begin{theo}
     \label{theomain1}
     Assume  {\bf (H1)-(H3)}.
     For every $k \in \mathbb{N}^*$ and for every compact interval $I_{0}$ in $]U_{-}, U_{+}[$, there exists a constant $C>0$ such that 
     for any initial data $\psi_{0} \in H^k$, the solution $\psi(t)$ to \eqref{symm-vort} satisfies the estimate
    $$ \| \langle A\rangle^{-k} g_{I_{0}}(H) \psi(t) \|  \leq  { C \over 1 + t^k} \| \langle A \rangle^k \psi_{0}\|,$$
uniformly in $t\ge 0$ and $\alpha \in \ZZ^*$,  where $A = i\partial_y$ and $g_{I_{0}}$ is any smooth and compactly supported function in $I_{0}$.
     \end{theo}

    The above result can be easily translated in the original velocity  coordinates. Indeed, for $\alpha \in \mathbb{Z}^*$, we have
    $$\| v^1_{\alpha}(t)\| \leq \|\langle A \rangle^{-1} \omega_{\alpha} (\alpha t)\|
    =  \|\langle A \rangle^{-1} m S^{-1} \psi  (\alpha t)\| \lesssim  \| \langle A \rangle^{-1} 
      \psi(\alpha t)\|$$
      since $\langle A \rangle^{-1} m S^{-1}\langle A\rangle$ is a bounded operator.
      In a similar way, we have
      $$  \| v^2_{\alpha}(t)\| \lesssim    |\alpha | \| \langle A \rangle^{-2} 
      \psi(\alpha t)\|.$$
      Therefore, we obtain the following
    \begin{cor}\label{cor-inviscid}
     Assume that the initial vorticity is of the form   $\omega_{0}= mS^{-1} g_{I_{0}}(H) \psi_{0}$, for any $\psi_0 \in H^2$ and
      for any compact interval $I_0$ in $]U_{-}, U_{+}[$. Then, the solution $v$ to \eqref{Euler1} satisfies the following estimates
    $$  \| v^1_{\alpha}(t)\| \leq {C \over 1 + (|\alpha| t)} \|\psi_{0}\|_{H^1}, \quad
      \| v^2_{\alpha}(t)\| \leq {C |\alpha|  \over 1 + (|\alpha| t)^2} \|\psi_{0}\|_{H^2},$$
uniformly in $t\ge 0$ and $\alpha \in \ZZ^*$, with $v_\alpha$ being the Fourier transform of $v$ with respect to variable $x$.     \end{cor}
 
The fact that the finite edges of the spectrum of $H$ at $U_{\pm}$ are not covered is a well-known
 limitation of the Mourre's theory \cite{Mourre}. In our case, this is a real difficulty that  comes from the fact that $U'(y)$ tends to zero at infinity and hence there is no positive lower bound in the Mourre's estimate \eqref{estmourre}.

 \section{Uniform mixing  and enhanced dissipation}
  We shall now describe our results for the viscous equations \eqref{NS1}. Again we write the equation in the vorticity
   form  and take the Fourier transform in $x$, leading to 
   \beq
\label{NS1-a}
 \partial_{t} \omega  + i \alpha L_{0}(y) \omega - \nu \Delta_{\alpha} \omega = 0
 \eeq
 where $L_{0}$ is defined in \eqref{Euler-a}. As in the inviscid case, we shall focus on the case $\alpha \neq 0$. Let us again set $\omega = mS^{-1} \psi$ to obtain  for $\psi$
$$
\partial_{t} \psi+ i \alpha H\psi- \nu S { 1 \over m} \Delta_{\alpha} m S^{-1} \psi= 0.
$$
Note that we shall not perform  the time scaling \eqref{change time}, as it is not well adapted to the viscous term. In addition, the equation is no longer symmetric. 
Nevertheless, it is symmetric up to a very small  error. Precisely, we can write the above equation under the form
\beq
\label{NSpsi}
\partial_{t} \psi+ i \alpha H\psi- \nu  \Delta_{\alpha }\psi=  \nu R\psi
\eeq
in which 
\beq
\label{Rpsidef} 
\begin{aligned}
R & = S { 1 \over m} \Delta_{\alpha} m S^{-1} - \Delta_\alpha  = S {1\over m} \partial_y^2 m S^{-1} - \partial_y^2 
\\&=  S {m''\over m} S^{-1}  + 2 S {m' \over m} \partial_{y} S^{-1} + 
 (S- 1) \partial_{y}^2 S^{-1}  + \partial_{y}^2  S^{-1} (1 - S).
\end{aligned} \eeq
 As we will see, the right hand-side does not have much influence on the dynamics for times
  $\nu t \ll 1$.

We shall  use the form \eqref{NSpsi} to state our main results.
 At first we shall establish that the estimates of Theorem \ref{theomain1} can be generalized to 
 \eqref{NSpsi} up to the viscous dissipation time scale $\nu^{-1}$. Precisely, we have

\begin{theo}
\label{theomain2}
 Assume  {\bf (H1)-(H3)}.
     For every $k \in \mathbb{N}^*$ and for  every compact interval $I_{0}\subset]U_{-}, U_{+}[$, there exist positive constants $C, M_0$ such that 
     for every initial data $\psi_{0} \in H^k$ and every $\nu \in (0, 1]$, the solution $\psi(t)$ to \eqref{NSpsi} satisfies the estimate
  $$   \| \langle A\rangle^{-k} g_{I_{0}}(H) \psi(t) \| \leq C\left( {1 \over 1+ (|\alpha| t)^k} \| \langle A \rangle^k
      \psi_{0}\| +   (\nu t)^{1 \over 2} e^{ M_{0} \nu t} \| \psi_{0}\|\right), $$
 uniformly in $t\ge 0$ and $\alpha \in \ZZ^*$,  where $A = i\partial_y$ and $g_{I_{0}}$ is any smooth and compactly supported function in $I_{0}$.
\end{theo}

Note that the above result shows that the estimates of Theorem \ref{theomain1}
 remain valid up to a correction term that is very small as long as $\nu t \ll 1$.
 One can think that the study of the stability of stationary shear flows $U_s = [U(y),0]$ in the Navier-Stokes equation
      is not really pertinent for times larger than $\nu^{-1}$. Indeed, 
       $U_s(y)$ is not an exact stationary solution of the nonlinear Navier-Stokes equation, though it  is classical
       in fluid mechanics to add a small stationary  forcing term in the equation so that $U_{s}(y)$ becomes
        an exact solution (see \cite{Reid} for example). The exact shear solution of Navier-Stokes equations (without a forcing) is time dependent, $U_s = [U(t,y),0]$, with $U(t,y)$ solving the heat equation
        $$ \partial_{t} U - \nu \partial_{y}^2 U=0, \quad U_{\vert_{t=0}}= U(y).$$
         As long as $\nu t\ll1$, it does not make much a difference to replace $U(t,y)$ by $U(y)$.
          Nevertheless, for $\nu t\gtrsim1$, the stationary profile $U(y)$ is no longer
          a good approximation, and in particular the derivatives $\partial_{y}^lU(t,y)$
           are damped by the diffusion.  This was taken into account for example in the papers
           \cite{Zhang17c, MaeMas, LinXu}.
       Let us also point out that our assumptions
  {\bf (H1)-(H3)} ensure the spectral stability of the shear flows to the Euler equations, but no assumptions were made to ensure
   the stability to the Navier-Stokes equations for all times (noting that since the channel $\TT \times \RR$ has no boundary, 
    the result of \cite{Grenier-Guo-Nguyen,Grenier-Guo-Nguyen1} does not apply).
    
  Our last main result is the following local enhanced dissipation for \eqref{NSpsi}.
  
 \begin{theo}
\label{theomain3}
Assume  {\bf (H1)-(H3)}.
     For every compact interval $I_{0}\subset]U_{-}, U_{+}[$, there are positive constants $C_0, M_0, c_0$ such that 
     for every initial data $\psi_{0} \in H^1$ and every $\nu \in (0, 1]$, the solution $\psi(t)$ to \eqref{NSpsi} satisfies the estimate
   $$ N(t) \leq C_0 \Big( e^{-c_{0} \nu^{1 \over3 } t} N(0) 
    +( \nu^{1 \over 3}+ (\nu t)^{1 \over 2} e^{M_{0} \nu t}) ( \|\psi_{0}\| + \|\alpha \psi_0\|) \Big)$$
 uniformly in $t\ge 0$ and $\alpha \in \ZZ^*$,  where 
  $$N(t)= \| g_{I_{0}}(H) \psi(t) \| + \| \alpha g_{I_{0}}(H) \psi(t)\| + \nu^{1 \over 3} \| \partial_{y}  g_{I_{0}}(H) \psi(t)\| $$
   and $g_{I_{0}}$ is any smooth and compactly supported function in $I_{0}$.
   
\end{theo}

   From the above estimate we see that after localization 
    in a strict spectral subspace of $H$  the solution of \eqref{NSpsi} is damped 
    at the time scale $\nu^{-{1 \over 3}}$ which is much smaller than the usual viscous
    dissipation scale $\nu^{-1}$.

\section{Proof of the inviscid results}
\label{proofEuler} 
In this section, we shall prove the results stated in Section \ref{introEuler}.
\subsection{Proof of Lemma \ref{lemSigma}}
 First, we observe that the essential spectrum of $\Sigma=1 + m \Delta_{\alpha}^{-1} m$ on $L^2$ is reduced to $1$ because of the decay assumptions on $m$ 
  in   {\bf (H2)}. Thus, it suffices to show that $\Sigma$ has only positive eigenvalues. Let us assume by contradiction that  $\lambda \leq 0$ is an eigenvalue of $\Sigma$. That is, there exists a nonzero $\psi \in L^2$ such that
  $$ \Sigma \psi = \lambda \psi.$$
   Set  $u= \Delta_{\alpha}^{-1} m \psi$. Then, $u \in H^2$ and
   $$ -\lambda  (-\Delta_{\alpha} u)  + (\mathcal{L} + \alpha^2 )u= 0$$
   where $\mathcal{L}  = -\partial_y^2 - m^2$ as defined in \eqref{Lcurldef}.
    Taking the scalar product with $u$ and integrating by parts, we get from {\bf(H3)} that
    $$ -\lambda \|\nabla_\alpha u\|^2  + (\lambda_{0}+ \alpha^2) \|u\|^2 \leq 0$$
with $\nabla_\alpha = (\partial_y,i\alpha)^T$.  Since $-\lambda \geq 0$, $\alpha \in \ZZ^*$, and  $\lambda_{0}+ \alpha^2>0$, we get that  $u=0$, which 
    is a contradiction. Lemma \ref{lemSigma} follows. 

 \subsection{Proof of Lemma \ref{lemspectrum}}
 Since $H$  is a compact perturbation of the multiplication operator by $U(y)$, we first get that $\sigma_{ess}(H)= [U_{-}, U_{+}]$.
 To exclude eigenvalues and embedded eigenvalues we will  adapt   the arguments of \cite{LinXu, LinSIAM, Sattinger} for  the Rayleigh equation in bounded domains. To proceed, let $c\in \RR$ be an eigenvalue of $H$. That is, there exists a nonzero $\psi \in L^2(\mathbb{R}) $
  such that  $$H \psi = c \psi.$$ 
 
 \subsubsection*{Case 1: $c \in \mathbb{R}\backslash [U_{-}, U_{+}]$.} In view of \eqref{symm-vort}, we get that the vorticity $ \omega =  m S^{-1} \psi \in L^2$ and solves
  $$ (U-c) \omega=  U'' \Delta_{\alpha}^{-1} \omega.$$
Setting $\phi =  \Delta_{\alpha}^{-1} \omega$, we note that $\phi \in H^2$ and solves the Rayleigh equation 
   \beq
   \label{vp1}
    - \partial_{y}^2 \phi  + {U'' \over U-c} \phi=  - \alpha^2 \phi.
    \eeq
    Note that since $c \not \in [U_{-}, U_{+}]$,   $U''/( U-c)$ is not singular. This means that $-\alpha^2 <0$ is an eigenvalue of the one-dimensional  Schr\"odinger operator $- \partial_y^2  + {U'' \over U-c}$.
   Since the essential spectrum of this operator is $[0,  + \infty[$, we obtain that the bottom of the spectrum is an eigenvalue $\lambda \le -\alpha^2 <0$  
    and that the corresponding eigenvector can be taken positive. Therefore, there exists $v \in H^2$, $v>0$ such that
    \beq
    \label{vp2} - \partial_{y}^2 v  - \lambda v  = - {U'' \over U-c} v.
    \eeq
    Observe that $U''/( U-c)$ belongs to $L^1$, since $1/(U-c)$ is bounded and
       $U''= - Um^2 \in L^1$, upon recalling from Assumption {\bf(H2)} that $U$ is bounded and $m\in L^2.$
    By using the Green's function of  $ - \partial_{y}^2 - \lambda$, we get from \eqref{vp2} that
     $$ v =  - G_{\sqrt{- \lambda}} * { {U'' \over U-c} v }, \qquad   G_{\sqrt{- \lambda}} (y)= - { 1 \over  \sqrt{- \lambda}} e^{- 
      \sqrt{- \lambda} \, |y|}.$$
     In particular, $v\in L^1$, since  
     $$ \|v \|_{L^1} \lesssim  \left\|  { {U'' \over U-c} v }\right \|_{L^1} \lesssim \|v\|_{L^\infty} \lesssim \|v\|_{H^1}<+\infty.$$
   Finally, we rewrite \eqref{vp2} as 
    $$ -   \partial_{y}\left( (U-c) \partial_{y} v \right)  +\partial_{y}( U' v)  = \lambda (U-c) v.
   $$
     Therefore, we obtain after integration  that $\lambda \int_{\mathbb{R}} (U-c) v = 0$, which is a contradiction since $v>0$ and $U-c$ has a constant sign.
     
     \subsubsection*{Case 2: $c\in \{U_{-}, \, U_{+}\}$.} In this case, we have $U''/(U-c) \in L^\infty \cap L^2$ from Assumption {\bf(H2)}.  Hence, again we have $v\in L^1$, 
     since  $$ \|v \|_{L^1} \lesssim  \left\|  { {U'' \over U-c} v }\right \|_{L^1} \lesssim \|v\|_{L^2} <+\infty.$$
     We thus arrive at the same contradiction as in the previous case.

     \subsubsection*{Case 3: $c \in ]U_{-}, U_{+}[$.} Let $y_c$ be the point (which is unique since $U'>0$)
      such that $U(y_{c})= c$ and set  $ I_{-}= ]-\infty, y_{0}[$ and $ I_{+}= ]y_{0}, + \infty[$. As in the previous cases, we get that there
      exists  a nontrivial $\phi \in H^2(\mathbb{R})$  that solves the Rayleigh equation \eqref{vp1} on $I_{\pm}$.
      
       We first prove that we must have $\phi(y_{c})\neq 0.$ Indeed, assuming otherwise that $\phi(y_{c})=0$ and proceeding as above,  we get that at least one of the  self-adjoint operators
        $L_{\pm} = -\partial_y^2 + {U'' \over U-c}$  with domain $H^2(I_{\pm}) \cap H^1_{0}(I_{\pm})$ (which are well defined thanks to the Hardy inequality and the fact that $U'>0$)
         has a negative eigenvalue $- \alpha ^2$. Therefore, we again find that  for  one of the intervals  $I_{\pm}$, there exist a negative eigenvalue $\lambda_{\pm}$
          and a positive eigenfunction $v_{\pm}$ such that
        $$ -   \partial_{y}\left( (U-c) \partial_{y} v_{\pm} \right)  +\partial_{y}( U' v_{\pm})  = \lambda_{\pm} (U-c) v_{\pm}, \quad y \in I_{\pm}.$$
        We can then also  integrate on $I_\pm$ to obtain  
        $$  \lambda_{\pm} \int_{I_{\pm}} (U-c) v_{\pm} =0$$
        upon recalling that $U(y_{c})= c$ and $v_{\pm}(y_{c})=0$. This yields a contradiction, since $U-c$ and $v_{\pm}$ have a constant sign on $I_\pm$. This proves that $\phi(y_{c})\neq  0$.

        Next, since $\phi \in H^2(\mathbb{R})$ and
         solves \eqref{vp1}, we have
         $ {U'' \over U-c} \phi \in L^2$. Together with $\phi(y_{c}) \neq 0$, we must have $U''(y_{c})= 0$.
         Consequently, we have proven that if $c \in  ]U_{-}, U_{+}[$ is an embedded eigenvalue, we must have $c=U(y_{c})$ with $U''(y_{c})=0$.
          Since  we assume that $U''/U$ is strictly negative, we must also have $U(y_{c})=0$ and therefore the only
          remaining possibility  for an embedded eigenvalue is $c=0$. Going back to the expression of $H$ in \eqref{symm-vort},  
          we immediately see that $0$ is not an eigenvalue of $H$
           since $S$ is invertible thanks to Lemma \ref{lemSigma}.

 \subsection{Proof of Lemma \ref{lemmourre} }   
 We shall now turn to the proof of Lemma  \ref{lemmourre}.
Recall that $H = SUS$ with $S= (1 + m \Delta_{\alpha}^{-1} m )^{1 \over 2}$. Let us write   
   $$ H= U +  (S-1) U (S-1) + (S-1) U + U(S-1)$$
   in which we note that $S-1$ is a compact operator on $L^2$, upon noting that $(1+S)^{-1}$ is bounded, $m \Delta_\alpha^{-1}m$ is compact on $L^2$, and $S- 1=  m \Delta_{\alpha}^{-1}m (1 + S)^{-1}$. Take $A = i\partial_y$ as the conjugate operator. 
We obtain    $$ i[H, A] = U'  + K_{1} = F(U) + K_{1}$$
    with $F(U(y)) = U'(y)$ and $K_{1}$ a compact operator on $L^2$. 
    
    Let $I$ be a compact interval in $]U_-,U_+[$ and let $g$ be in $\mathcal{C}^\infty_{c} (]U_{-}, U_{+}[, \mathbb{R}_{+})$ with the support contained in $I$. 
We then take $\tilde I\subset ]U_{-}, U_{+}[$ to be a slightly bigger interval
     such that there exists a smooth $\tilde g$ with the support contained in $\tilde I $
      and $\tilde g=1$ on $I$.  Since $F$ is bounded below away from zero on the support of $\tilde g$, we get
      that there exists a positive constant $\theta_{I}$ such that
    \beq
    \label{mourre1} \tilde g(U)   i[H, A]  \tilde g(U) \geq \theta_{I} \tilde g(U)^2  +  \tilde g(U) K_{1}  \tilde g(U).
    \eeq
     We can then write
     \begin{multline*} g(H)  i[H, A]  g(H)=  g(H) \tilde g(H)   i[H, A] \tilde {g}(H) g(H) \\=
      g(H) \tilde g(U) i[H, A] \tilde {g}(U) g(H)  + g(H) \Bigl(( \tilde g(H) - \tilde g(U))  i[H, A] \tilde {g}(H)  +
       \tilde{g}(U)   i[H, A] (\tilde {g}(H) -g(U))\Bigr)g(H).
       \end{multline*}
       Thus, using \eqref{mourre1}, we get 
       $$   g(H) \tilde g(U) i[H, A] \tilde {g}(U) g(H)  \geq \theta_{I} g(H) \tilde{g}(U)^2 g(H)
        \geq \theta_{I} g(H)^2  + g(H)( \tilde g(U) ^2 - \tilde g(H)^2)  g(H).$$
         To conclude, it suffices to use that if $f\in \mathcal{C}^\infty_{c}(\mathbb{R})$, then
          $f (H) -f (U)$ is a compact operator.
     We refer to Lemma \ref{lemtech2}, ii).
 
\subsection{Local decay estimates}
  We shall now prove a propagation estimate that will be crucial for the proof of Theorem \ref{theomain1}. 
  \begin{lem}
  \label{lemprop} Let $I$ be a compact interval in $ ]U_{-}, U_{+}[$  such that Lemma \ref{lemmourre} holds. Let $J\subset I$ and $g_{J}$ be in $\mathcal{C}^\infty_{c} (]U_{-}, U_{+}[, \mathbb{R}_{+})$, having its support contained in $J$ and satisfying 
    \beq
    \label{hypprop}
     g_{J}(H) i [H, A] g_{J}(H) \geq { \theta_{I} \over 2} g_{J}(H)^2, 
    \eeq
with $\theta_I$ as in Lemma \ref{lemmourre}. Then, for every $k \in \mathbb{N}$, there exists a constant $C_k$ so that 
    \beq
    \label{estprop} \|  \langle A \rangle^{- k} g_{J}(H) \psi(t)  \| \leq {C_{k} \over \theta_{I}^{k+{1 \over 2}}} { 1 \over  1 + t^k} \|\langle A \rangle^k  g_{J}(H)\psi_{0}\|, 
\eeq
for every $t\ge 0$ and for every $\psi_{0} \in H^k$, where $\psi$ solves  \eqref{symm-vort}.
    \end{lem}
   
\begin{proof} Take $\chi(\xi)= {1 \over 2} (1 - \tanh \xi)$ and observe that $\chi$ has the property that 
\beq
\label{propchi}  \chi'= - \phi^2, \qquad |\phi^{(m)} (\xi) | \leq C_{m} \phi(\xi) , \qquad \forall \xi \in \mathbb{R}, \, \,
\forall m \in \mathbb{N}^*
\eeq
where $\phi = 1/(\sqrt2 \cosh \xi)$. Following the method of \cite{HSS}, we shall use a localized energy estimate. Set $A_{t,s}= \frac1s ( A - a - \theta t )$ for $A = i\partial_y$, $a \in \mathbb{R}$, $s \geq 1$ and
 $\theta = {\theta_{I} \over 4}.$ In what follows, $\chi$ and $\phi$ stand for
 $\chi(A_{t,s})$, $\phi(A_{t,s})$, respectively, and $g_{J}$ for $g_{J}(H)$. These are self-adjoint  operators on $L^2$, and $g_{J}$
  commutes with $H$. In addition, all the estimates are uniform in $a$  and $s \geq 1$, and 
  they do not depend on the subinterval $J$. 
  
  Using the equation \eqref{symm-vort} and symmetry properties, we observe that 
 \beq
 \label{est1} {d \over dt}  \|  \chi^{1\over 2} g_{J} \psi \|^2= {d \over dt } \langle \chi g_{J}
  \psi, g_{J} \psi\rangle= { \theta \over s} \|  \phi g_{J} \psi \|^2
    + \langle i[H,\chi] g_{J} \psi, g_{J} \psi \rangle.
  \eeq
  To evaluate the right hand side, we use the commutation formula from \cite{HSS,Golenia,GerardC}, which we recall in Lemma \ref{lemcom1}. For every $p\ge 1$, we get 
    \beq
  \label{com1}  \langle   i[ H,\chi]  f, f \rangle= - {1 \over s} \langle \phi^2 i [H,A] f,  f \rangle   +   \sum_{j=2}^{p-1} {1 \over j!} {1 \over s^j} \langle  \chi^{(j)}
   i \,ad_{A}^jH f, f \rangle  +{1 \over s^p} \langle R_{p}f, f \rangle
   \eeq
   with $ad_AH = [H,A]$, $ad_A^jH = [ad_A^{j-1}H,A]$, and  
      $$\|R_{p}\| \leq C_{p} \| ad_{A}^pH \| \leq C_{p}$$
   where $\| \cdot \|$ stands here for the operator norm from $L^2$ to $L^2$.
   In the next computation, we continue to denote by $R_{p}$ any bounded operator which is bounded
   by a harmless constant.
    For the first term on the right of \eqref{com1}, we use again the commutation formula to get
   $$\begin{aligned} & { 1 \over s } \langle \phi^2 i [H,A] f,  f \rangle= {1 \over s}\langle \phi i[H,A] f, \phi f \rangle
   \\&=
  {1 \over s } \langle i  [H,A] \phi f, \phi f\rangle +  \sum_{j=1}^{p-2}  { 1 \over j!} { 1 \over s^{j+1}} \langle \phi^{(j)} 
     i \,ad_{A}^{j+1} Hf, \phi f \rangle  + {1 \over s^{p}} \langle R_{p}f, f \rangle.
     \end{aligned}$$
    For the terms in the above  sum, we can use repeatedly the commutation formula to get in the end that
    $$  { 1 \over s } \langle \phi^2 i [H,A] f,  f \rangle= {1 \over s}\langle i[H,A] \phi  f, \phi f \rangle 
    + \sum_{j=1}^{p-2}{ 1 \over s^{j+1}} \sum_{k,l} \langle R_{k} \phi_{l}f, \phi f\rangle + {1 \over s^{p}} \langle R_{p}f, f \rangle$$
    where in the above sum $k,\,l$ runs in finite sets and  $\phi_{l}$ stands for  some  derivatives of $\phi$, which in particular satisfies the estimate
     $|\phi_{l}| \lesssim |\phi|$ by using \eqref{propchi}. 
     
     In a similar way, to estimate the other terms in \eqref{com1},  we observe that $\chi^{(j)} = -(\phi^2)^{(j-1)}$ can be expanded as a sum of terms under the form
     $ \phi_{k} \tilde \phi_{m}$ where $\phi_{k}$, $\tilde \phi_{m}$ and their derivatives are controlled by $\phi.$
      By using again the commutation formula as many times as necessary, this allows to write
       an expansion under the form
      $$  \sum_{j=2}^{p-1} {1 \over j!} {1 \over s^j} \langle  (\phi^2)^{(j-1)}
   ad_{A}^jH f, f \rangle=  \sum_{j=2}^{p-1}  { 1 \over s^j} \sum_{k,l,m} \langle R_{k} \phi_{l} f, \,  \tilde  \phi_{m}f
    \rangle + { 1 \over s^p } \langle R_{p}f, f\rangle.$$
    In particular, we get from \eqref{com1} and the above expansion formula that for every $f$
     (assuming $s \geq 1$)
 \beq
 \label{RHS1}
  \langle   i[ H,\chi]  f, f \rangle  \leq - { 1 \over s} \langle i[H,A] \phi f , \phi f\rangle
    + { C_{p}  \over s^2} \| \phi f \|^2 + {C_{p} \over s^p }  \|f\|^2.
 \eeq
 From \eqref{est1}, we thus find that
\begin{align*}  {d \over dt}  \|  \chi^{1\over 2} g_{J} \psi  \|^2  & \leq
 {1 \over  s} \Big (\theta \|\phi g_{J} \psi\|^2 -   \langle i[H,A] \phi g_{J}\psi , \phi g_{J} \psi\rangle\Big ) 
  +  {C_{p}  \over s^2 }\| \phi g_{J} \psi \|^2 + {C_{p} \over s^p }  \|g_{J} \psi\|^2  \\
   & \leq { 1 \over s} \left( \theta - {\theta_{I} \over 2}\right)   \| \phi g_{J} \psi \|^2+ 
    {C_{p}  \over s^2} \| \phi g_{J} \psi \|^2 + {C_{p} \over s^p }  \|g_{J} \psi\|^2 
    \end{align*}
    where we have used \eqref{hypprop} in the last inequality. 
     Consequently, we can choose $\theta = \theta_{I}/4$ and $s$ sufficiently large
     ($ s\geq  { 16C_{p} \over  \theta_{I}}$) to obtain
   \beq\label{est-loc} {d \over dt} \|  \chi^{1\over 2} g_{J} \psi(t) \|^2   \leq {C_{p} \over s^p }  \|g_{J} \psi (t)\|^2
    \leq {C_{p} \over s^p }  \|g_{J} \psi_{0}\|^2\eeq
    upon using  
    $$ {d \over dt } \|g_{J}(H) \psi\|^2= 0.$$
     Integrating \eqref{est-loc} between $0$ and $t$ and recalling $\chi = \chi(A_{t,s})$, we find that for every $t$, 
   $$ \left\| \chi^{1 \over 2} \left( {A - a - \theta t \over s} \right)g_{J} \psi(t) \right\|^2
    \leq  \left\| \chi^{1 \over 2} \left( {A - a  \over s}\right) g_{J} \psi_{0} \right\|^2 +  {C_{p} t \over s^p } \|g_{J} \psi_{0}\|^2$$
uniformly for all $a \in \mathbb{R}$ and $s \geq 1$, with $\theta = {\theta_{I} \over 4}$. In particular for $\theta_{I}t \geq { 1 \over \theta_{I}^2}$,  we can take $s= C_{p} (\theta_{I} t)^{ 1 \over 2}$ and
  $a =  -{ \theta_{I} \over 8} t$ to obtain
  \beq
  \label{est2} \left\| \chi^{1 \over 2} \left( {A -  {\theta_{I}  \over 8} t \over C_{p}(\theta_{I} t)^{1 \over 2}} \right)g_{J} \psi(t) \right\|^2
    \leq  \left\| \chi^{1 \over 2} \left( {A  +  { \theta_{I} \over 8} t  \over  C_{p}(\theta_{I} t)^{1 \over 2}}\right) g_{J}\psi_{0} \right\|^2 +   { C_p \over \theta_{I} (\theta_{I}t)^{{p \over 2} - 1 } }\|g_{J} \psi_{0}\|^2.
    \eeq 

  To conclude, for $k\ge 0$, we write that
  \begin{equation}\label{est-finalX} \| \langle A\rangle^{-k} g_{J} \psi(t) \| \leq
    \left\| \langle A\rangle^{-k} \chi^{1 \over 2} \left( {A -  {\theta_{I}  \over 8} t \over C_{p}(\theta_{I} t)^{1 \over 2}} \right) g_{J} \psi(t)\right\| + \left\| \langle A\rangle^{-k} \left(1- \chi^{1 \over 2} \left( {A -  {\theta_{I}  \over 8} t \over C_{p}(\theta_{I} t)^{1 \over 2}} \right) \right) g_{J} \psi(t)\right\| .
    \end{equation}
Let us first estimate the second term on the right. 
By using $\|g_{J}(H) \psi (t)\|
     =  \|g_{J}(H) \psi_{0}\|$, it suffices to bound in the operator norm
    $$  \left\| \langle A\rangle^{-k} \left(1- \chi^{1 \over 2} \left( {A -  {\theta_{I}  \over 8} t \over C_{p}(\theta_{I} t)^{1 \over 2}} \right) \right) \right\| \lesssim { 1 \over (\theta_{I} t )^k }.$$
    Indeed, the estimate is clear, when $A \geq \theta_{I}t /16$, due to the factor $\langle A\rangle^{-k} $. In the case when $A \leq \theta_{I}t /16$, we observe that $1-\chi^{1 \over 2}$ term can be bounded by 
     $e^{-   C(\theta_{I} t)^{1 \over 2}}$, which is again bounded by the algebraic decay. 
     
     Let us now bound the first term on the right of \eqref{est-finalX}. 
     Using \eqref{est2} and choosing $p$ sufficiently large, we thus get
  $$ \| \langle A\rangle^{-k} g_{J} \psi(t) \| \lesssim 
       \left\|  \chi^{1 \over 2} \left( {A  +  { \theta_{I} \over 8} t  \over  C_{p}(\theta_{I} t)^{1 \over 2}}\right) \langle A\rangle^{-k} \langle A\rangle^{k} g_{J} \psi_{0} \right\|  + { 1 \over {\theta_{I}^{1 \over 2}}} { 1 \over (\theta_{I}t)^k}  \|g_{J} \psi_{0}\|.
     $$
In the above, the first term on the right is bounded by $C_p (\theta_{I} t )^{-k} \|\langle A\rangle^{k} g_{J} \psi_{0}\|$ by considering $A\le -\theta_I t/16$ and $A\ge -\theta_I t/16$ and using the fact that $\chi(\xi)$ decays exponentially to zero as $\xi \to +\infty$.

          Thus, we have obtained 
           $$  \theta_{I}^{k+{ 1 \over 2 }} \| \langle A\rangle^{-k} g_{J} \psi(t) \| \lesssim {1 \over  1 + t^k} \| \langle A \rangle^k g_{J}
        \psi_{0}\|$$ for $\theta_{I}^3t \geq 1$. 
        The estimate for $\theta_I^3t \leq 1$ is clear. The lemma follows.
        \end{proof}

   \subsection{Proof of Theorem \ref{theomain1}}
   \label{finEuler}
    Let us take $I_{0}$ any closed interval included in $]U_{-}, U_{+}[$ and take $I$ such that  $I_{0} \subset \mathring{I}$ 
    and that Lemma \ref{lemmourre} holds. In particular, for every point $E \in I_{0}$ and every positive number $\delta$, we can take $g_{E, \delta}$
     a smooth function supported in $]E-2 \delta , E+ 2 \delta[$ and equal to one on $[E- \delta, E+ \delta].$ 
      For $g_{E, \delta}(H)$ and for $\delta$ small enough so that $]E-2 \delta , E+ 2 \delta[ \subset I$, Lemma \ref{lemmourre} yields 
$$
g_{E, \delta}(H) i [H, A] g_{E, \delta}(H) \geq \theta_{I} g_{E, \delta}(H)^2 + g_{E, \delta}(H) K g_{E, \delta}(H). $$
Let us show that
we can take $\delta$
       sufficiently small such that \eqref{hypprop} holds for $ g_{E, \delta}(H)$. Indeed, since $K$ is compact, we can approximate
       it by a finite rank operator in the operator norm. Thus, it suffices to prove that for every $\eps>0$, 
       we get 
        $g_{E, \delta} K g_{E, \delta} \geq - \eps g_{E, \delta}^2  $, for sufficiently small $\delta$ 
        and for $ K= a\otimes b$ a rank one operator. In this case, we then have
      $$ \langle g_{E, \delta} K g_{E, \delta} f, f \rangle= 
         \langle  g_{E, \delta} f , a\rangle \langle g_{E, \delta} f, b\rangle, $$
         and therefore by Cauchy-Schwarz
       $$ \langle g_{E, \delta} K g_{E, \delta} f, f \rangle \geq - \|g_{E, 2 \delta} a\| \|g_{E, 2\delta}b\|\,  \|g_{E, \delta} f\|^2$$
       where $g_{E, 2 \delta}$ is a smooth function supported in $]E-4 \delta , E+ 4 \delta[$
        that is one on the support of $g_{E, \delta}$.
          The result  follows by using that for  $c= a, \, b \in L^2$, thanks to the spectral measure, we can write
          $$   \| g_{E,  2 \delta} c\|^2 =  \int_{[U_{-}, U_{+}]}  |g_{E,  2 \delta}(\lambda)|^2  \langle dE_{\lambda}c, c\rangle$$
          and  by using the Lebesgue theorem, upon noting that the measure  $ \langle dE_{\lambda}c, c\rangle$ is continuous, 
         thanks to Lemma  \ref{lemspectrum}. 
         This proves that  \eqref{hypprop}, and hence, \eqref{estprop} hold for $J = ]E-2 \delta , E+ 2 \delta[$. 
   
       Finally, we can cover $I_{0}$ by a finite  number of such intervals $J$ with $\overline{J} \subset I$ sufficiently small  such that
        \eqref{hypprop} holds. Take a partition of unity associated to this covering of $I_{0}$. For each
         $J$, the estimate \eqref{estprop} holds (noting that the constants in the estimate are independent
          of $J$).  Taking an initial data under the form $g_{I_{0}} (H) \psi_{0}$ supported in $I_{0}$, we can
then  sum the estimate to obtain the final result, Theorem \ref{theomain1}.  Note however that the constants in the final estimate do 
           depend on $I_{0}$ and might blow up at the edges of the spectrum of $H$.

\section{Viscous case}
\label{proofNS}
We shall now prove Theorem \ref{theomain2} and Theorem \ref{theomain3}.
 We use the form \eqref{NSpsi} of the equation.
  To estimate the remainder $R$ defined as in \eqref{Rpsidef}, 
 we can use  again that both $S, S^{-1}$ are bounded operators and 
 \begin{align*}
 &   S- 1 = m \Delta_{\alpha}^{-1}m (1+S)^{-1} =  m \Delta_{\alpha}^{-1}m (   1 + ( 1 +  m \Delta_{\alpha}^{-1}m)^{\frac12})^{-1}.
 \end{align*}
Thus, in view of \eqref{Rpsidef}, we can write  
\beq
\label{Rnu1}
R = R^0+ \partial_y R^1, \qquad 
\| R^0\| + \| R^1\| \leq C_{R} 
\eeq
for some constant $C_{R}$ that is independent of $\nu$.

\subsection{Basic energy estimate}

 As a preliminary, we first establish that
 
  \begin{prop}
  \label{propNS1}
  There are positive constants $M_{0}, C$ such that for every $\nu \in (0, 1]$, 
  the solution of \eqref{NSpsi} satisfies the  estimates
  \beq
  \label{0visc}   \| \psi(t) \| \leq e^{ M_{0}\nu t } \|\psi_{0}\|, 
   \qquad \nu \int_{0}^t \| \nabla_\alpha\psi \|^2 \leq \|\psi_{0}\|^2 ( 1  + C\nu t e^{ 2M_{0} \nu t})
   \eeq 
   uniformly for all $t\ge 0$ and $\alpha\in  \ZZ$. Here, $\nabla_{\alpha}=(\partial_{y}, i\alpha)^T$. 
  \end{prop}
  Note that the above estimates are uniform in $\alpha$. In addition, when $\alpha$ is large enough, the estimates can be improved in the sense that we could take  $M_{0}=0$. However, we shall not use the improvement.

  \begin{proof}
  The proposition is an easy consequence of the fact that $H$ is symmetric.
  Indeed, taking integration by parts and using  \eqref{NSpsi} and \eqref{Rnu1}, we obtain that
  $$ {1 \over 2} {d \over dt} \| \psi\|^2 + \nu \| \nabla_{\alpha} \psi\|^2  \leq C \nu ( \| \psi \|^2 +  \|\psi \|\, \|\partial_{y}\psi\|).$$
  Using the Young inequality, we thus get 
  \begin{equation}\label{basicEE}  {d \over dt} \| \psi \|^2 + \nu \| \nabla_{\alpha} \psi\|^2  \leq C \nu  \| \psi \|^2.\end{equation}
The first estimate in \eqref{0visc} follows from the Gronwall inequality, while the second is obtained by integrating in time the above inequality. 
     \end{proof} 
    \subsection{Proof of Theorem \ref{theomain2}}
    We proceed as in the proof of Theorem \ref{theomain1}.
  We first choose $I_{0}$ and $I$ as in Section \ref{finEuler} and cover $I_{0}$ with a finite number of  small 
  intervals such that on each  small interval the estimate {\eqref{estprop} holds. 
   Let us take $J$ to be any of these small intervals. We now proceed as in the proof of Lemma \ref{lemprop}
    by computing
    $$ {d \over dt} \| \chi^{1\over 2} g_{J} \psi \|^2$$
    with $g_{J}= g_{J}(H)$ and
    $$ \chi= \chi(A_{\alpha t, \alpha s}), \qquad  A_{\alpha t, \alpha s}=  {A - a - \theta \alpha t \over \alpha s}.$$
    Note that the only difference here is that we have replaced $t$ and $s$ by $\alpha t$ and $\alpha s$, since we
    did not perform the change of the time scale as in \eqref{change time}. We again focus on $\alpha>0$.
   As similarly done in \eqref{est1}, we obtain for the solution $\psi(t)$ to \eqref{NSpsi} 
   \begin{equation}
   \label{C0}
\begin{aligned}  &   {d \over dt}  \|  \chi^{1\over 2} g_{J} \psi\|^2= {d \over dt } \langle \chi g_{J}
  \psi, g_{J} \psi\rangle 
  \\&={ \theta \over s} \|  \phi g_{J} \psi \|^2
    +  \langle i\alpha [ H,\chi] g_{J} \psi, g_{J} \psi \rangle      +2 \nu \langle  \chi g_{J} \Delta_{\alpha} \psi, g_{J} \psi \rangle + 2\nu \langle \chi g_{J} R \psi, g_{J} \psi\rangle .
      \end{aligned}\end{equation}
We now estimate each term on the right. The first two terms are estimated  exactly as done in the proof of Lemma \ref{lemprop} or precisely in \eqref{RHS1}, yielding 
\beq
\label{CH} { \theta \over s} \|  \phi g_{J} \psi \|^2
    +  \langle i\alpha [ H,\chi] g_{J} \psi, g_{J} \psi \rangle    \leq  {1 \over s} \left( \theta  - {\theta_{I} \over 2 } + {C_{p} \over \alpha s}\right) \| \phi g_{J} \psi \|^2 + 
 { \alpha C_{p} \over (\alpha s)^p } \|g_{J} \psi \|^2
 \eeq
 where $C_{p}$ is independent of $\alpha$ and $s$ (and $\nu$, of course).
 Next for the third term on the right of \eqref{C0}, we can integrate by parts  (observe that $\chi$ commutes with $\partial_y$) to obtain
 \begin{align*}
   \langle  \chi g_{J} \Delta_{\alpha} \psi, g_{J} \psi \rangle
   & = - \| \nabla _{\alpha} \chi^{1\over 2} g_{J} \psi \|^2 +
  \langle  \chi^{1 \over 2} [g_{J}, \partial_{y}] \partial_{y} \psi, \chi^{1 \over 2} g_{J} \psi \rangle 
   - \langle \chi^{1 \over 2} [g_{J}, \partial_{y}] \psi, \chi^{1 \over 2} \partial_{y} g_{J} \psi  \rangle \\
    & =  - \| \nabla _{\alpha} \chi^{1\over 2} g_{J} \psi \|^2 + \langle \chi^{1 \over 2} \left[ [g_{J}, \partial_{y} ], \partial_{y}\right] \psi, 
     \chi^{1 \over 2} g_{J} \psi \rangle -  2 \langle \chi^{1 \over 2} [g_{J}, \partial_{y}] \psi, \chi^{1 \over 2} \partial_{y} g_{J} \psi  \rangle .
  \end{align*}
   Using the Lemma \ref{lemtech2} i) to estimate the commutators, we find
  $$
    \langle  \chi g_{J} \Delta_{\alpha} \psi, g_{J} \psi \rangle   \leq   - \| \nabla _{\alpha} \chi^{1\over 2} g_{J} \psi \|^2 + C \| \psi \| (  \|\chi^{1 \over 2 }g_{J} \psi \| +  \| \partial_{y} \chi^{1 \over 2 }g_{J} \psi \|) .$$
   In a similar way, using the decomposition \eqref{Rnu1} and integrating by parts, we get 
 $$  \langle \chi g_{J} R \psi, g_{J} \psi\rangle  \leq C \| \psi \|  (  \| \partial_{y} \chi^{1 \over 2 }g_{J} \psi \|
    + \|\chi^{1 \over 2 }g_{J} \psi \|).$$
Using $\|  \chi^{1 \over 2 }g_{J}\|\lesssim 1$ and the Young inequality, we thus obtain 
   \beq
   \label{Cnu}
    \langle  \chi g_{J} \Delta_{\alpha} \psi, g_{J} \psi \rangle  +  \langle \chi g_{J} R \psi, g_{J} \psi\rangle   \leq   -{ 1 \over 2 } \| \nabla _{\alpha} \chi^{1\over 2} g_{J} \psi \|^2 + C \| \psi \|^2
      \eeq
      for some constant $C$ that is independent of $\nu$.
      
 Consequently, putting \eqref{CH} and \eqref{Cnu} into \eqref{C0}, and choosing again $\theta = \theta_{I} /4$
 and $s$ large so that $\alpha s\ge 4C_p/\theta_I$, we obtain 
 \beq
 \label{visc10} 
 \begin{aligned}  {d \over dt}  \|  \chi^{1\over 2} g_{J} \psi\|^2
 &\leq  {\alpha C_{p} \over( \alpha s)^p}
  \|g_{J} \psi \|^2 + C \nu \|\psi \|^2 
 \leq {\alpha C_{p} \over( \alpha s)^p}
  \|g_{J} \psi \|^2  + C\nu  e^{2 M_{0} \nu t} \|\psi_{0}\|^2 
\end{aligned}  \eeq
  where the last estimate comes from Proposition \ref{propNS1}. 
  On the other hand, using the same  commutator estimates as above (now with $\chi =1$), we also get that
  \beq
  \label{gJdecay}  {d \over dt} \|  g_{J} \psi \|^2  + \nu \| \nabla_{\alpha} g_{J} \psi  \|^2 \leq C \nu \|\psi \|^2 \le  C\nu  e^{2 M_{0} \nu t} \|\psi_{0}\|^2
  \eeq
which, after an integration in time, yields 
%
$\|g_{J} \psi (t) \| \leq \|g_{J} \psi_{0}\| + C (\nu t)^{1 \over 2} e^{M_{0} \nu t} \| \psi_{0}\|.$
Hence, the inequality \eqref{visc10} now becomes  
\beq
 \label{visc10-X} 
 \begin{aligned}  {d \over dt}  \|  \chi^{1\over 2} g_{J} \psi\|^2
 \leq {\alpha C_{p} \over( \alpha s)^p}
  \|g_{J} \psi_0\|^2  + C\nu  e^{2 M_{0} \nu t} \|\psi_{0}\|^2 .
\end{aligned}  \eeq

Finally, for times $t$ such that $\theta^3_{I} \alpha t \geq 1$, we integrate \eqref{visc10-X} over $(0,t)$ and take
   $ \alpha s= C_{p} (\alpha \theta_{I} t)^{1 \over 2} $ and $ a= -{ \theta_{I} \alpha t \over 8}$. Recalling $\chi= \chi(A_{\alpha t, \alpha s})$, we obtain
\begin{multline*}
   \left\| \chi^{1 \over 2} \left( {A -  {\theta_{I}  \over 8} \alpha  t \over C_{p}(\theta_{I}\alpha t)^{1 \over 2}} \right)g_{J} \psi(t) \right\|^2
    \leq  \left\| \chi^{1 \over 2} \left( {A  +  { \theta_{I} \over 8} \alpha t  \over  C_{p}(\theta_{I} \alpha t)^{1 \over 2}}\right) g_{J}\psi_{0} \right\|^2 + { C_{p} \alpha t\over  (\theta_{I} \alpha t)^{{p \over 2} } }\|g_{J} \psi_{0}\|^2
     +  C \nu t e^{2 M_{0} \nu t} \|\psi_{0}\|^2.
     \end{multline*}
 Note that in this estimate $C_{p}$ is independent of $J$ and $\theta_{I}$, while $C$ might depend on the compact intervals $I_0$ and $I$. 
 From this estimate, we easily deduce in the same way   as done in the proof of  Lemma \ref{lemprop}
  that
  $$  \theta_{I}^{k+{ 1 \over 2 }} \| \langle A\rangle^{-k} g_{J} \psi(t) \| \lesssim {1 \over  (\alpha t)^k} \| \langle A \rangle^k g_{J}
        \psi_{0}\| +  C_{N} (\nu t)^{1 \over 2} e^{M_{0} \nu t} \|\psi_{0}\|^2,$$
for times $t$ so that  $\theta^3_{I} \alpha t \geq 1$. When  $\theta^3_{I} \alpha t \le 1$, the estimate is clear. Thus, summing up over a finite number of such small intervals $J$, we complete the proof.
 \subsection{Proof of the enhanced dissipation, Theorem \ref{theomain3}}
 We again consider a finite number of small intervals $J$ covering $I_{0}$  as above so that 
     \beq
    \label{hypprop-re}
     g_{J}(H) i [H, A] g_{J}(H) \geq { \theta_{I} \over 2} g_{J}(H)^2, 
    \eeq
for $g_{J}\in \mathcal{C}^\infty_{c} (]U_{-}, U_{+}[, \mathbb{R}_{+})$ with support contained in $J$. 
We shall estimate $g_J(H) \psi$, with $\psi$ solving \eqref{NSpsi}. We compute   
\beq
\label{eqgjpsi}
\partial_{t} g_{J} \psi + i \alpha H g_{J} \psi - \nu \Delta_{\alpha}g_{J} \psi = 
  \nu \mathcal{C}_{\nu}\psi
 \eeq
 where $ \mathcal{C}_{\nu} = \mathcal{C}_{\nu}^0 + \partial_{y} \mathcal{C}_{\nu}^1$, 
 with
 \begin{align*}
   \mathcal{C}_{\nu}^0 =\left[ [g_{J}, \partial_{y}], \partial_{y}\right] + g_{J}R^0 + [g_J,\partial_y] R^1,  \qquad  \mathcal{C}_{\nu}^1
 = 2  [g_{J}, \partial_{y}]  +  g_{J} R^1,
 \end{align*}
 Here, $R^0,R^1$ are as in \eqref{Rnu1}. In particular, from the commutator estimates, we obtain
  \beq
  \label{comnufin}
  \| \mathcal{C}_{\nu}^0\| + \| \mathcal{C}_{\nu}^1 \| \leq C .
  \eeq


Take again $A= i \partial_{y}$. The starting point is to compute
 $$ 
 \begin{aligned} -{ 1 \over 2} {d \over dt }\langle A g_{J} \psi, \alpha g_J\psi \rangle
  &= - \alpha \langle  \partial_{t} g_{J} \psi, A g_{J}  \psi \rangle
   \\&= \alpha^2 \langle i H  g_{J} \psi, A g_{J} \psi \rangle -  \nu \alpha  \langle \Delta_{\alpha} g_{J} \psi , Ag_{J} \psi\rangle
    - \nu  \alpha \langle \mathcal{C}_{\nu} \psi, A g_{J} \psi \rangle. 
    \end{aligned}$$ 
 The crucial term in the above identity is the first one on the right hand-side. Indeed, thanks to \eqref{hypprop-re}, we have
 $$  \alpha^2 \langle i H  g_{J} \psi, A g_{J} \psi \rangle   = - { 1 \over 2} \alpha^2
  \langle i[H, A]  g_{J}\psi,  g_{J}\psi \rangle  \leq - { \theta_{I} \over 4} \alpha^2 \| g_{J} \psi \|^2.$$
  For the viscous terms on the right hand-side, we estimate 
  $$  \nu |\alpha| | \langle \Delta_{\alpha} g_{J} \psi , Ag_{J} \psi\rangle| 
   \lesssim  \nu  \|\Delta_\alpha g_{J} \psi \|\| \alpha\partial_{y} g_{J} \psi \|$$
   and after an integration by parts
   $$ \nu  |\alpha| \, |\langle \mathcal{C}_{\nu} \psi, A g_{J} \psi \rangle|
    \lesssim \nu  \| \alpha \psi  \| \, (\| \partial_{y} g_{J} \psi \| + \| \partial_{y}^2 g_{J}\psi \|).$$ 
 This yields 
\begin{equation}\label{key1}
 \begin{aligned} - {d \over dt } & \langle A g_{J} \psi, \alpha g_J\psi \rangle + { \theta_{I} \over 2}  \| \alpha g_{J} \psi \|^2
 \\ &\lesssim \nu  \|\Delta_\alpha g_{J} \psi \|\| \alpha \partial_{y} g_{J} \psi \| 
  + \nu  \| \alpha \psi  \| \, (\| \partial_{y} g_{J} \psi \| + \| \partial_{y}^2 g_{J}\psi \|). 
    \end{aligned}\end{equation}

Next, using \eqref{gJdecay}, we have 
 \beq
 \label{gJ2}
{d \over dt}  \| \alpha g_{J} \psi(t) \|^2 + \nu \| \alpha \nabla_{\alpha} g_{J} \psi (t) \|^2  \leq  C\nu \|\alpha \psi\|^2.
 \eeq

It remains to estimate $\|A g_{J} \psi \|^2$. Similarly as done above, we get 
$$ \frac12 {d \over dt } \|A g_{J} \psi \|^2
 + \nu  \| \nabla_{\alpha} \partial_{y} g_{J}\psi \|^2
  \lesssim \nu   ( \| \psi \| + \|\partial_{y} \mathcal{C}_{\nu}^1 \psi \|) \|\partial_{yy} g_{J} \psi \| 
   + \alpha \| [H,A] g_{J} \psi \| \| \partial_{y} g_{J}\psi \| $$
 in which 
   $$  \| [H,A] g_{J} \psi \| \lesssim \|g_{J}\psi \|, \quad   \|\partial_{y} \mathcal{C}_{\nu}^1 \psi \| \lesssim 
    \|\psi \| + \| \partial_{y} \psi \|.$$
Thus, using the Young inequality, we obtain
   \beq
   \label{gJ4}  {d \over dt }  \|A g_{J} \psi \|^2
 + \nu  \| \nabla_{\alpha} \partial_{y} g_{J}\psi \|^2
  \lesssim \nu ( \|\psi \|^2 + \| \partial_{y}\psi \|^2) + \|\alpha g_J \psi\| \| \partial_y g_J \psi\|.
  \eeq

  To conclude, we shall combine the estimates \eqref{gJdecay}, \eqref{key1}, \eqref{gJ2}, and \eqref{gJ4}
   in a suitable way. We introduce 
    $$ Q(t)= \Gamma^4 ( \| g_{J}\psi (t) \|^2 +  \|\alpha g_{J} \psi (t) \|^2 )
      -  \Gamma  \nu^{1 \over 3 }   \langle A g_{J} \psi, \alpha g_{J} \psi \rangle + \nu^{2 \over 3} \|A g_{J} \psi  (t)\|^2$$
       where $\Gamma\geq 1$ is a large parameter (independent of $\nu$ and $\alpha$) that we will choose later.
        We first observe that if $\Gamma$ is sufficiently large, $Q(t)$ is equivalent to a weighted
         $H^1$ norm. Namely,   
       $$ Q(t) \approx \|g_{J}\psi (t)\|^2 + \|\alpha g_{J} \psi(t)\|^2 +  \nu^{2 \over 3} \| \partial_{y} g_{J} \psi(t)\|^2.$$
We now add up the estimates \eqref{gJdecay}, \eqref{key1}, \eqref{gJ2}, and \eqref{gJ4} with the corresponding weight as in $Q(t)$ and use the Young inequality to obtain  
\begin{equation}\label{key2-ineq}    {d \over dt } Q(t) + c_{0} \nu^{1 \over 3}  Q(t)
  \le C_0 \nu e^{ 2 M_{0}  \nu t } ( \|\psi_{0}\|^2 + \|\alpha \psi_0\|^2)+ C_0  \nu^{\frac53} \| \partial_{y}\psi \|^2\end{equation}
  for some positive constants $C_0,c_0$. Indeed, the left hand side is clear, upon recalling that $|\alpha|\ge 1$. Let us check the right hand side. 
 In view of \eqref{key1}, we estimate  
  $$
  \begin{aligned}  \Gamma \nu^{\frac43}  \|\Delta_\alpha g_{J} \psi \| \| \alpha \partial_{y} g_{J} \psi \| & \le \Gamma^{-1} \nu^{\frac53}  \| \partial_y^2 g_{J} \psi \|^2 + \nu \|\alpha^2 g_J \psi\|^2 + C_0 \Gamma^2 \nu (\Gamma + \nu^{\frac23}) \| \alpha \partial_{y} g_{J} \psi \| ^2
  \\
 \Gamma  \nu^{\frac43}  \| \alpha \psi  \| \, (\| \partial_{y} g_{J} \psi \| + \| \partial_{y}^2 g_{J}\psi \|) 
& \le \Gamma^{-1} \nu^{\frac53}  \| \partial_y^2 g_{J} \psi \|^2 + \nu \|\partial_y g_J \psi\|^2 + C_0 \Gamma^2 \nu (\Gamma + \nu^{\frac23}) \| \alpha \psi \| ^2
\end{aligned}  $$
in which each term on the right, except the last term involving $\|\alpha\psi\|^2$, is absorbed into the left hand side, precisely the corresponding viscous term, of \eqref{gJ4}, \eqref{gJ2}, and \eqref{gJdecay}, upon taking $\Gamma$ large enough. Similarly, in view of \eqref{gJ4}, we estimate 
$$ \nu^{\frac23}\|\alpha g_J \psi\| \| \partial_y g_J \psi\| \lesssim \nu^{\frac13} \|\alpha g_J \psi \|^2 + \nu \|\partial_y g_J \psi\|^2$$
which is again controlled by the left hand side of \eqref{key1} and the viscous term in \eqref{gJdecay}, respectively. Thus, we have obtained 
$$
 {d \over dt } Q(t) + c_{0} \nu^{1 \over 3}  Q(t)
  \lesssim  \nu ( \|\psi \|^2 + \|\alpha \psi\|^2) + \nu^{\frac53}\| \partial_{y}\psi \|^2 .$$
This yields \eqref{key2-ineq}, upon using \eqref{0visc}.

   Finally, we integrate the differential inequality \eqref{key2-ineq} and use \eqref{0visc} again to obtain
     $$ Q(t) \leq e^{ - c_{0} \nu^{1 \over 3 } t} Q(0) + C_0 ( \nu^{2 \over 3}  + C  \nu t e^{2 M_{0} \nu t})( \|\psi_{0}\|^2 + \|\alpha \psi_0\|^2).$$  
Theorem \ref{theomain3} follows.

  \section{Technical lemmas}
   \label{sectiontech} 
  In this section, we shall recall some commutator estimates used 
   throughout the paper. These results can be found, for instance, in \cite{GerardC, Golenia, HSS}.
    The main idea is to use the Helffer-Sj\"ostrand formula to express the functional calculus of a self-adjoint
    operator.
   
   Let us start with almost analytic extensions. Let us introduce $S^\rho$ for $\rho \in \mathbb{R}$ the
   set of $\mathcal{C}^\infty$ functions on $\RR$ such that  
   $$ |f^{(m)}(x)| \leq C_{m} \langle x \rangle^{\rho-m},\qquad \forall~x\in \RR, \quad \forall~ m \in \mathbb{N}.$$
We also set 
   $$ \|f\|_{\rho}= \sup_{x\in \RR, \, m\in \NN} \langle x \rangle^{m-\rho} |f^{(m)}(x)|.$$
   An almost analytic extension of $f$  is a function $\tilde f$  on $\mathbb{C}$ such that 
 \begin{equation}  \begin{aligned}
   \label{propftilde}
   & \tilde f_{\mathbb{R}}= f, \\
  &  \mbox{supp } \tilde f \subset\Big\{ x+iy, \, |y| \leq 2 \langle x \rangle, \, x \in \mbox{supp }f \Big\}, \\
  & | \partial_{\overline{z}} \tilde f(z)| \leq C \langle x \rangle^{\rho - N -1}|y|^N
   \end{aligned}
   \end{equation}
   for some $N$ fixed and large enough.
   As an example, one can take   $$\tilde{f}(x+ i y)=\left( \sum_{r=0}^Nf^{(r)} (x)  { (iy)^r \over r !} \right) \chi\left({y \over \langle x \rangle}\right)$$
   where $\chi(s)$ is a smooth function which is equal to $1$ for $|s| \leq 1$ and $0$ for $|s| \geq 2$.
   
   Now, let $T$ be a self-adjoint operator. For any $f\in S^\rho$, we define the operator $f(T)$ 
   by   \beq
    \label{HS}
     f(T) = \lim_{R\rightarrow + \infty} {i \over 2 \pi} \int_{\mathbb{C} \cap |Re \, z| \leq R}
      \partial_{\overline{z}}\tilde f(z) (z- T)^{-1} \, dL(z)
      \eeq
      where $dL(z)= dx dy $ is the Lebesgue measure on $\mathbb{C}$  identified to $\mathbb{R}^2$.
      Observe that when $\rho<0$,  the above  integral converges in the operator norm.
  
  \begin{lem}
  \label{lemcom1}\cite{GerardC, Golenia,HSS}
  For $k\ge 1$, let $f \in S^\rho$ with $\rho<k$, and let $B$ be a bounded self-adjoint operator on $L^2$ such that the iterated commutators
   $ad_{T}^jB$, $j \leq k$, are also bounded. Then, there holds the expansion 
  $$ [B,f(T)]=  \sum_{j=1}^{k-1} {1 \over j!} f^{(j)}(T)ad_{T}^jB + R_{k}(f,T,B)$$
  with 
  $$ \|R_{k}(f,T,B)\| \leq C_k(f) \|ad_{T}^kB\|$$
  where $C_k(f)$ depends only on $k$ and $\|f\|_\rho$.
  \end{lem}

In addition, we also use the following:
   \begin{lem}
   \label{lemtech2} Let $f \in \mathcal{C}^\infty_{c}(\mathbb{R)}$, $A= i \partial_{y}$, and $H$ be the bounded and self-adjoint operator defined as in \eqref{symm-vort}. 
    Then, we have
    \begin{enumerate}
    \item[(i)] $ [A, f(H)]$ is a bounded operator.
    \item [(ii)] $f(H)- f(U)$ is a compact operator.
      \end{enumerate}
      \end{lem}
      \begin{proof}
      We start with proving (i). Thanks to \eqref{HS} for $f(H)$, we have
      $$ [ A, f(H) ]={i \over 2 \pi} \int_{\mathbb{C}}  \partial_{\overline{z}} \tilde f(z) \, (z-H)^{-1} [A, H] (z-H)^{-1}  dL(z), $$
      with $\tilde f$ an almost analytic extension of $f$.
      Since $[A, H]$ is bounded, the result follows directly from the facts that $f$ is compactly supported and that the integral converges in the operator
      norm thanks to \eqref{propftilde}.
      
      Let us next prove ii). In a similar way, we write
      $$ f(H)- f(U)= {i \over 2 \pi}\int_{\mathbb{C}}  \partial_{\overline{z}} \tilde f(z) \,( (z-H)^{-1} - (z-U)^{-1}) \, dL(z).$$
      Since $H= U+K$ with $K$ compact, the above yields 
 $$  f(H)- f(U)= {i \over 2 \pi}\int_{\mathbb{C}}  \partial_{\overline{z}} \tilde f(z) \,(z-H)^{-1}
    K (z-U)^{-1}\, dL(z).$$   
   Again, the integral converges in the operator norm, since  $(z-H)^{-1}
    K (z-U)^{-1}$  is a compact operator for every $z \notin \mathbb{R}$. The result follows.
           \end{proof}

\subsection*{Acknowledgement} Part of this work was done when the first three authors were visiting ICERM, Brown University, T. Nguyen was visiting Universit\'e Paris-Sud at Orsay under a visiting professorship, and T. Nguyen and A. Soffer were visiting CCNU (C.C. Normal Univ.)  Wuhan, China. The authors thank the institutions for their hospitality and the support. A.S. is partially supported by NSF grant DMS 01600749 and NSFC 11671163.


\bibliographystyle{abbrv}

\end{document}